\documentclass{amsart}
\usepackage{amsmath, amssymb,epic,graphicx,mathrsfs,enumerate}
\usepackage[all]{xy}

\usepackage{amsthm}
\usepackage{amssymb}
\usepackage{latexsym}
\usepackage{longtable}
\usepackage{epsfig}
\usepackage{amsmath}
\usepackage{hhline}

\DeclareMathOperator{\F}{F}

 \DeclareMathOperator{\frat}{Frat}

\DeclareMathOperator{\GL}{GL}

\DeclareMathOperator{\End}{End}

\newcommand{\N}{\mathbb N}

\newtheorem{thm}{Theorem}
\newtheorem{cor}[thm]{Corollary}
 \newtheorem{lemma}[thm]{Lemma}
\newtheorem{prop}[thm]{Proposition} 
 \newtheorem{defn}[thm]{Definition}
\newtheorem{question}[]{Question} 
\numberwithin{equation}{section}

\renewcommand{\footnote}{\endnote}
\newcommand{\ignore}[1]{}\makeglossary

\begin{document}
	\bibliographystyle{amsplain}
	\title[Intersection properties in prosolvable groups]{Intersection of maximal subgroups\\ in prosolvable groups}
\author{Iker de las Heras}
	\author{Andrea Lucchini}
	
	\address{Universit\`a degli Studi di Padova\\  Dipartimento di Matematica \lq\lq Tullio Levi-Civita\rq\rq\\ Via Trieste 63, 35121 Padova, Italy, lucchini@math.unipd.it, delasherasiker@gmail.com}

\begin{abstract} Let $H$ be an open subgroup of a profinite group that can be expressed as intersection of maximal subgroups of $G.$ Given a positive real number $\eta,$ we say that $H$ is an $\eta$-intersection 
if  there exists a family of maximal subgroups $M_1,\ldots,M_t$ such that $H=M_1\cap\ldots\cap M_t$ and $|G:M_1|\cdots|G:M_t|\le |G:H|^{\eta}$. We investigate the meaning of this property and its influence on the group structure.

	\end{abstract}
	\maketitle
	
\section{Introduction}
	
Let $G$ be a (topologically) finitely generated profinite group. We write $c_n(G)$ to express the number of (open) subgroups of $G$ of index $n$ which are intersections of maximal subgroups of $G.$ We will say that $c_n(G)$ is polynomially bounded if there exists $\beta$ independent of $n$ such that $c_n(G)\le n^\beta$. Mann asked the following question \cite[Problem 4]{PFG}: \begin{question}\label{quedue}What are the groups for which $c_n(G)$ is polynomially bounded?
	\end{question}  This question is related with the discussion of a conjecture proposed by Mann in the same paper. A profinite group $G$ is positively finitely generated (PFG) if for some finite $k,$ a random $k$-tuple of elements generates $G$ with probability $P(G,k)>0.$ Mann conjectured that if $G$ is PFG, then the Dirichlet series $$P(G,s) = \sum_H \mu(H,G) |G:H|^{-s},$$ defines an analytic function on some right half-plane of $\mathbb C$ and takes the values $P(G,k)$ for (sufficiently large) $k\in \mathbb N$ (here $H$ ranges over the lattice of all open subgroups of $G$ and $\mu$ is the M\"{o}bius function associated to this lattice). To establish this, it is sufficient to verify that $P(G,s)$ converges absolutely in some right half-plane; this is the case if and only if                                 \begin{enumerate}
	\item $|\mu(H,G)|$ is bounded by a polynomial function of $|G:H|;$  \item the number $b_n(G)$ of open subgroups $H$ of index $n$ satisfying $\mu(H,G)\neq 0$ grows at most polynomially in $n.$
\end{enumerate}
As it is noticed in \cite{PFG} p. 447,  if $H$ is an open subgroup of $G$ and $\mu(H,G)\neq 0$, then $H$ is an intersection of maximal subgroups, hence $b_n(G)\leq c_n(G)$ and  condition (2) is satisfied if $c_n(G)$ is polynomially bounded. 
A celebrated result of Mann and Shalev \cite{pmsg} 
states that a profinite group is PFG if and only if it has polynomial
maximal subgroup growth (PMSG). Since $\mu(M,G)=-1$ for any maximal subgroup $M$ of $G,$ it must be $m_n(G)\leq b_n(G)\leq c_n(G)$ (where $m_n(G)$ denotes the number of
maximal subgroups of $G$ with index $n$). In particular, if $c_n(G)$ grows polynomially, then $G$ has PMSG. A natural question that arises from these considerations is the following: \begin{question}Is $c_n(G)$ polynomially bounded when $G$ has PMSG?\end{question}
In \cite{AL}, Mann's conjecture about the absolutely convergence of	 $P(G,s)$ has been proved in the particular case when $G$ is a finitely generated prosolvable groups (recall that every finitely generated prosolvable group $G$ has PMSG \cite[Theorem 10]{PFG}). In that paper it is proved that if $G$ is a finitely generated prosolvable, then $b_n(G)$ is polynomially bounded but is still unknown whether $c_n(G)$ is also polynomially bounded. In this paper, we want to investigate whether the methods employed in \cite{AL} to prove that $b_n(G)$ is polynomially bounded, can be adapted to study the behaviour of $c_n(G).$ The main observation in \cite{AL} is that if $G$ is a prosolvable group and $H$ is an open subgroup of $G$ with $\mu(H,G)\neq 0,$ 
then the maximal subgroups $M_1,\ldots,M_t$ such that 
  $H=M_1\cap\ldots\cap M_t$  can be chosen with the extra property that $|G:M_1|\cdots|G:M_t|=|G:H|.$ On the other hand, what is really useful to bound  the sequence $b_n(G)$ is not the equality
  $|G:M_1|\cdots|G:M_t|=|G:H|$ but that  existence of a constant $\eta$, independent on the choice of $H,$ such that  $|G:M_1|\cdots|G:M_t|\leq|G:H|^\eta.$
This suggests the following definitions:

\begin{defn}\label{defzero}
	Let $G$ be a finite group and let $\eta$ be a positive real number. We say that a maximal intersection $H$ in $G$ is an $\eta$-intersection   if  there exists a family of maximal subgroups $M_1,\ldots,M_t$ such that:
	\begin{enumerate}
		\item $H=M_1\cap\ldots\cap M_t$.
		\item $|G:M_1|\cdots|G:M_t|\le |G:H|^{\eta}$.
	\end{enumerate}
\end{defn}

\begin{defn}\label{defuno}
	Let $G$ be a finite group and let $\eta$ be a positive real number. We say that $G$ has the $\eta$-intersection property  if every maximal intersection in $G$ is an $\eta$-intersection.
\end{defn}

\begin{defn}\label{defdue}
	Let $G$ be a profinite group. We say that $G$ has the bounded intersection property if there exists a positive real number $\eta$ such that  every open subgroup $H$ of $G$ which is an intersection of maximal subgroups of $G$ is an $\eta$-intersection.
\end{defn}

Clearly a profinite group $G$ has the bounded intersection property if and only if there exists a positive real number $\eta$ such that $G/N$ has the $\eta$-intersection property for every open normal subgroup $N$ of $G.$

\

The connection between the previous  definitions 
and Question \ref{quedue} is clarified by the following Proposition.

\begin{prop}\label{propo}Suppose that $G$ is a profinite group  with polynomial
	maximal subgroup growth. If $G$ has the bounded intersection properties, then there exists a constant
		$\beta$ such that $c_n(G) \leq n^\beta.$	\end{prop}

Other two definitions will play a relevant role in our investigation of the bounded intersection property.

\begin{defn}
Let $G$ be a finite group, $V$ an irreducible $G$-module, $F=\End_G(V)$ and $\gamma$ a positive integer. We say that $V$ is a $\gamma$-module if, for every $F$-subspace $W$ of $V,$ there exists an $F$-subspace $W^*$ of $V$ such that:
\begin{enumerate}
	\item $\dim_F(W^*)\leq \gamma.$
	\item $C_G(W)=C_G(W^*)\cap (\cap_{M\in \mathcal M_W}M)$ where $\mathcal M_W$ is the set of the maximal subgroups of $G$ containing $C_G(W).$
\end{enumerate}
\end{defn}

\begin{defn}
	Let $G$ be a profinite group. We say that $G$ has the bounded chief factors property if the exists a positive integer $\gamma$ such that every complemented chief factor of $G$ is a $\gamma$-module.
\end{defn}

Let $M$ be a maximal subgroup of $G$ and denote by $Y_M=\bigcap_{g\in G}M^g$ the normal core of $M$ in $G$ and by $X_M/Y_M$ the socle of the primitive permutation group $G/Y_M$ (in its action on the right cosets of $M/Y_M$ in $G/Y_M$): clearly $X_M/Y_M$ is a chief factor of $G$ and $M/Y_M$ is a complement of $X_M/Y_M$ in $G/Y_M.$ Now assume that $H$ is a maximal intersection in $G$: we will denote by $\mathcal V_H$ the set of the irreducible $G$-module which are isomorphic to $X_M/Y_M$ for some maximal subgroup $M$ of $G$ containing $H.$ We will prove:

\begin{thm}\label{thuno} Let $G$ be a finite solvable group and let $H$ be a maximal intersection in $G$. If  every $V\in \mathcal V_H$ is a $\gamma$-module, then $H$ has the  $(\gamma+1)$-intersection property.
\end{thm}

\begin{cor}
Let $G$ be a finite solvable group. If there exists $\gamma$ such that every complemented chief factor of $G$ is a $\gamma$-module, then $G$ has the $(\gamma+1)$-intersection property.
\end{cor}

We will also prove a converse result:

\begin{thm}\label{due}
	Let $G$ be a finite solvable group. If $G$ has the $\eta$-intersection property, then every complemented chief factor is a $\lfloor \eta\cdot c\rfloor$-module, with $c\backsimeq 3.243,$ the P\'{a}lfy-Wolf constant.
\end{thm}

In particular we deduce:

\begin{thm}\label{equivalenza}A prosolvable group $G$ has the bounded intersection property if and only if it has the bounded chief factors property.
	\end{thm}
	
Clearly if $V$ is an irreducible $G$-module and $\dim_{\End_G(V)}V\leq \gamma,$ then $V$ is a $\gamma$-module. So the following corollary is an immediate consequence of Proposition \ref{propo} and Theorem \ref{equivalenza}.
\begin{cor}
	Let $G$ be a finitely generated prosolvable. If there exists $\gamma\in \mathbb N$ such that $\dim_{\End_G(V)}V\leq \gamma$ for every irreducible $G$-module $G$-isomorphic to a complemented chief factor of $G,$ then $c_n(G)$ is polynomially bounded.
\end{cor}	
	
An application of the previous result, is the following:
	
\begin{cor}\label{fittingamma}
	Let $G$ be a finitely generated prosolvable group. If the derived subgroup is  pronilpotent, then $c_n(G)$ is polynomially bounded.
\end{cor}

In particular, we may apply  Corollary \ref{fittingamma} to the class of the finitely generated prosupersolvable groups. However, even if the number of intersections of  maximal subgroups in a finitely generated prosupersolvable group $G$ grows polynomially with respect to the index, it may happen that the amount of such subgroups is really ``big'' comparing with the number of subgroups of $G$ with non-zero M\"obius number. Let us denote by $\tilde \beta_n(G)$ the number of conjugacy classes of subgroups with index at most $n$ and  with non-zero M\"obius number and by $\tilde \gamma_n(G)$ the number of conjugacy classes of subgroups with index at most $n$ that are intersection of maximal subgroups. In Section \ref{super} we will construct a 2-generated prosupersolvable group with the properly that $$\liminf_{n\to \infty}\frac{\tilde \beta_n(G)}{\tilde \gamma_n(G)}=0.$$
We don't know examples of prosolvable groups that don't satisfy the bounded chief factors property. A positive answer to the following intriguing question, will imply that all the prosolvable groups have the bounded chief factors property.
\begin{question} Does there exists a constant $\gamma$ such that, for every finite solvable group $G$, all the irreducible $G$-modules are $\gamma$-modules?
\end{question}

Notice that we can strengthen the definition of $\gamma$-module, setting that an irreducible $G$-module $V$ is a strong $\gamma$-module if, for every $\End_G(V)$-subspace $W$ of $V,$ there exists an $\End_G(V)$-subspace $W^*$ of $V$ such $\dim_{\End_G(V)}(W^*)\leq \gamma$ and $C_G(W)=C_G(W^*)$.

\begin{question} Does there exists a constant $\gamma$ such that, for every finite solvable group $G$, all the irreducible $G$-modules are strongly $\gamma$-modules?
\end{question}
\section{Preliminary results}

In this section we will give the proof of  Proposition \ref{propo} and we will recall the main properties of the crowns of a finite solvable group. Moreover we will start to study the maximal  intersections in a relevant case.

	\begin{proof}[Proof of Proposition \ref{propo}]
	Since  $G$ has PMSG,
		there exists $\alpha$ such that, for each $n \in \mathbb N$, the
		number of maximal subgroups of $G$ with index $n$ is bounded by
		$n^\alpha.$ 
		Now, for $n\neq 1$, we want to count the number of subgroups $H$ with $|G:H|=n$ which are intersections of maximal subgroups. By assumption there exists an integer $\eta$ with the property that, if we fix such an $H\le G$, then there is a family of maximal subgroups $M_1,\ldots, M_t$ such that $H=\cap_{1\le i\le t}M_i$ and $n_1\ldots n_t\le n^{\eta}$, where $n_i=|G:M_i|$. There are at most
		$$1+2+\ldots +n^{\eta}=\frac{n^{\eta}(n^{\eta}+1)}{2}$$
		possible factorizations of positive integers $\le n^{\eta}$ (see \cite{mado}), and for each fixed factorization $n_1\ldots n_t$, there are at most $n_i^{\alpha}$ choices for the maximal subgroup $M_i$ of index $n_i$. Therefore, there are at most $n_1^{\alpha}\ldots n_t^{\alpha}\le n^{\eta\cdot \alpha}$ choices for the family $M_1,\ldots,M_t$, and we conclude that
		$$c_n(G)\le \frac{n^{\eta}(n^{\eta}+1)}{2}n^{\eta\cdot \alpha}.$$
		Obviously, we always can find a constant $\beta$ such that
		$$\frac{n^{\eta}(n^{\eta}+1)}{2}n^{\eta\cdot \alpha}\le n^{\beta},$$
		for any $n\ge 1$, so the proof is complete.		
	\end{proof}

Let $G$ be a finite solvable group and let $M$ be a maximal subgroup of $G$ and denote by $Y_M=\bigcap_{g\in G}M^g$ the normal core of $M$ in $G$ and by $X_M/Y_M$ the socle of the primitive permutation group $G/Y_M$: clearly $X_M/Y_M$ is a chief factor of $G$ and $M/Y_M$ is a complement of $X_M/Y_M$ in $G/Y_M.$ Let $\mathcal M$ be the set of  maximal subgroups of $G,$ let $\mathcal V$ be a set of representatives of the irreducible $G$-modules that are $G$-isomorphic to some chief factor of $G$ having a complement and,
for every $V\in \mathcal V,$  let $\mathcal M_V$ be the set of  maximal subgroups $M$ of $G$ with $X_M/Y_M\cong_G V.$

Recall some results by Gasch\"utz \cite{gaz}. Given $V\in\mathcal V$, let $$R_G(A)  =\bigcap_{M\in\mathcal M_V}M.$$
It turns out that $R_G(A)$ is the smallest normal subgroup contained
in $C_G(A)$ with the property that $C_G(A)/R_G(A)$ is
$G$-isomorphic to a direct product of copies of $A$ and it has a
complement in $G/R_G(A)$. The factor group $C_G(A)/R_G(A)$ is
called the $A$-crown of $G$. The non-negative integer
$\delta_G(A)$ defined by $C_G(A)/R_G(A)\cong_G A^{\delta_G(A)}$ is
called the $A$-rank of $G$ and it coincides with the number of
complemented factors in any chief series of $G$ that are
$G$-isomorphic to $A$ (see for example \cite[Section 1.3]{classes}). In particular $G/R_G(A)\cong A^{\delta_G(A)}\rtimes H,$ with $H\cong G/C_G(A).$

\begin{lemma}{\cite[Lemma 1.3.6]{classes}}\label{corona}
	Let $G$ be a finite solvable group with trivial Frattini subgroup. There exists
	a crown $C/R$ and a non trivial normal subgroup $D$ of $G$ such that $C=R\times D.$
\end{lemma}

\begin{lemma}{\cite[Proposition 11]{crowns}}\label{sotto} Assume that $G$ is a finite solvable group with trivial Frattini subgroup and let $C, R, D$ be as in the statement of Lemma \ref{corona}. If $HD=HR=G,$ then $H=G.$
\end{lemma}

For a fixed $V\in\mathcal V,$ we want to study which subgroups of $G$ can be obtained as intersection of maximal subgroups in $\mathcal M_V.$ Since $R_G(V)\leq M$ for every $M\in \mathcal M_V,$ we are indeed asking which subgroups of the semidirect product  $G/R_G(A)\cong V^{\delta_G(H)}\rtimes G/C_G(V),$ can be obtained as intersection of maximal supplements of the socle $V^{\delta_G(H)}.$
So assume that $H$ is a solvable group acting irreducibly and faithfully on an elementary abelian $p$-group $V$ and, for a positive integer $t,$ consider the semidirect product $G=V^t\rtimes H$, where
 we assume that the action of $H$ is diagonal on $V^t,$ that is, $H$ acts in the same way on each of the direct factors.	The aim of the remaining part of this section is to describe which subgroups of $G$ can be obtained as intersections of maximal subgroups of $G$ supplementing $V^t.$ Notice that   $M$ is a maximal in $G$ supplementing $V^t$ if and only if $M=W H^v$ with $W$ a maximal $H$-submodule of $V^t$ and $v\in V^t$.
 
\begin{lemma}\label{interKM}
	Let $G=V^t\rtimes H$ be a finite solvable group such that $V$ is a faithful irreducible $H$-module. Let $K=W_1X$ and $M=W_2H^{v_2}$,
	where $W_1$ is a proper $H$-submodule of $V^t$, $W_2$ is maximal submodule of $V^t,$  $X\leq H^{v_1}$ for some $v_1\in V^t$ and $v_2\in V^t.$
	\begin{enumerate}[i)]
	\item If $W_1+W_2=V^t,$ then there exists $w_1\!\in\! W_1$ such that $K\cap M=(W_1\cap W_2)X^{w_1}.$
	\item If $W_1\leq W_2,$ then there exists $u\in V$ such that $K\cap M=W_1C_X(u).$
	\end{enumerate}
\end{lemma}

\begin{proof}
	Assume $W_1+W_2=V^t$. There exist $w_1\in W_1$ and $w_2\in W_2$ such that $v_1-v_2=w_2-w_1$, or equivalently, $v_1+w_1=v_2+w_2$.  Thus, we have
	$$X^{w_1}\le (H^{v_1})^{w_1}=H^{v_1+w_1}=H^{v_2+w_2},$$
	In particular
	$$\begin{aligned}(W_1\cap W_2)X^{w_1}&\leq W_1X^{w_1}=W_1X=K,\\
	(W_1\cap W_2)X^{w_1}&\leq (W_1\cap W_2)H^{v_2+w_2}\leq W_2H^{v_2+w_2}
	= W_2H^{v_2}=M.
	 \end{aligned}$$
Moreover, let $y_1x=y_2h$ be an element of $K\cap M$ with $y_1\in W_1$, $x\in X^{w_1}$, $y_2\in W_2$ and $h\in H^{v_2+w_2}$. We have $y_1-y_2=hx^{-1}\in V^t\cap H^{v_2+w_2}=1$, so it follows that $y_1=y_2$ and $h=x$. Hence, $K\cap M\le(W_1\cap W_2)X^{w_1}$, and since we proved above the other inclusion, we have the equality, as we wanted.
	
	Assume now that $W_1\le W_2$. We observe that
	$$K\cap M=W_1X\cap W_2H^{v_2}=W_1X\cap W_2X\cap W_2H^{v_2}.$$
	If $W_2X\le W_2H^{v_2}$, then
	$$K\cap M=W_1X\cap W_2X\cap W_2H^{v_2}=W_1X\cap W_2X=W_1X=K$$
	and we are done. So, we may assume $W_2X\not\le W_2H^{v_2}$. This, in particular, implies $W_2H^{v_1}\neq W_2H^{v_2}$ and consequently, $v_2-v_1\not\in W_2$.
	Since $V^t$ is a completely reducible $H$-module, there exists  an $H$-submodule $U$ of $V^t$ such that $V^t=W_2\times U$. Thus, there exists a non-trivial element $u\in U$ such that $v_2-v_1=w+u$ with $w\in W_2$. Hence, $H^{v_2-w}=H^{v_1+u}$ and 
	$$M=W_2H^{v_2}=W_2H^{v_2-w}=W_2(H^{v_1})^u.$$
		Let us show that $K\cap M=W_1C_X(u)$. Note that $C_X(u)=C_X(u)^u\le (H^{v_1})^u$, so it is clear that $W_1C_X(u)\le K\cap M$. In order to prove the other inclusion let $w_1x=w_2h^u$ be an element of $K\cap M$ with $w_1\in W_1$, $w_2\in W_2$, $x\in X$ and $h\in H^{v_1}$. Note that $w_1x=w_2h^u=w_2[u,h^{-1}]h$, and since $H$ normalizes $U$, it follows that $[U,H]\le U$. So, $w_1-w_2-[u,h^{-1}]=hx^{-1}\in V^t\cap H^{v_1}=1$, or in other words, $w_1-w_2=[u,h^{-1}]$ and $h=x$. Moreover, $w_1-w_2=[u,h^{-1}]\in W_2\cap U=0$, so $w_1=w_2$ and $[u,h^{-1}]=[u,h]=[u,x]=0$. Therefore, $x\in C_X(u)$ and $K\cap M\le W_1C_X(u)$, so that $K\cap M =W_1C_X(u)$, and we are done.
\end{proof}


\begin{thm}\label{impor}
	Let $G=V^t\rtimes H$ be a finite solvable group such that $V$ is a faithful irreducible $H$-module. Let $M_1=W_1H^{v_1},\ldots,M_n=W_nH^{v_n}$ be maximal subgroups of $G$ supplementing $V^t$, with $W_i$ a maximal $H$-submodule of $V^t$ and $v_i\in V^t$ for every $i$. Then, we have
	$$\bigcap_{i=1}^nM_i=U C_{H^*}(Z)$$
	where $U=\bigcap_{i=1}^nW_i$, $H^*$ is a conjugate of $H$ and $Z$ is an $\End_{H}(V)$-subspace of $V$.
\end{thm}

\begin{proof}
 Set $U_j:=\bigcap_{1\le i\le j} W_i$. Reordering the maximal subgroups, we may assume
$$V^t> U_1> U_2>\ldots> U_{t^*}=\ldots=U_n=U$$
for a suitable $1\le t^*\le n$. First we prove, by induction on $j$, that for every $j\leq t^*$ there exists $w_j\in V^t$ such that 
$$M_1\cap \dots \cap M_j=U_jH^{w_j}.$$ This is clear if $j=1.$ Assume $1<j\leq t^*.$ By induction $K:=M_1\cap \dots \cap M_{j-1}=U_{j-1} H^{w_{j-1}}$ for some $w_{j-1}\in V^t.$ Since $U_j=U_{j-1}\cap W_j<U_{j-1},$ we have $U_{j-1}+W_j=V^t$ and therefore, by Lemma \ref{interKM},
$$M_1\cap \dots \cap M_{j}=K\cap M_j=U_jH^{v_j}$$ for some $w_j\in V^t.$ In particular 
$$M_1\cap\dots\cap M_{t^*}=UH^*$$
with $H^*$ a suitable conjugate of $H$ and $U\cong_H V^{t-t^*}.$ Now consider $0\leq i\leq n-t^*.$ We prove, again by induction, that, for every $0\leq i\leq n-t^*,$ there exists an $\End_H(V)$-subspace $Z_i$ of $V$ with
	$$M_1\cap \dots \cap M_{t^*}\cap \dots\cap M_{t^*+i}=UC_{H^*}(Z_i).$$
When $i=0,$ we just take $Z_0=\{0\}.$ If $i>0,$ then by induction there exists $Z_{i-1}$ such that
$K=M_1\cap \dots \cap M_{t^*}\cap \dots\cap M_{t^*+i-1}=U C_{H^*}(Z_{i-1}).$ Since $U\leq W_{t^*+i},$ it follows from Lemma \ref{interKM} that there exists $z_i\in V$ such that
$$M_1\cap \dots\cap M_{t^*+i}=K \cap M_{t^*+i}=MC_{C_{H^*}(Z_{i-1})}(z_i)=MC_{H^*}(Z_i),$$ being 
$Z_i$ the $\End_H(V)$-subspace of $V$ spanned by $Z_{i-1}$ and $z_i.$
\end{proof}

\begin{thm}\label{possibile}
	Let $G=V^t\rtimes H$ be a finite solvable group such that $V$ is a faithful irreducible $H$-module. Assume that $U$ is an intersection of maximal $H$-submodules of $V^t$,  $H^*$ is a conjugate of $H$ and $Z$ is an $\End_{H}(V)$-subspace of $V.$ If $V^t/U\cong_H V^{t^*}$ and $d=\dim_{\End_{H}(V)}Z,$ then $UC_{H^*}(Z)$ can be obtained as intersection of $t^*+d$ maximal subgroups of $G$ supplementing $V^t.$
\end{thm}
\begin{proof}
By assumption, there exists $t^*$ maximal $H$-submodules $W_1,\dots,W_{t^*}$ of $V^t$ such that $U=W_1\cap \dots \cap W_{t^*}.$ Let $z_1,\dots,z_d$ be an $\End_H(V)$-basis of $Z.$ Let $A$ be a maximal $H$-submodule of $V^t$ containing $U:$ there exists  
an $H$-submodule $B$ of $V^t$ such that $V^t= A\times B.$
It must be $B\cong_H V,$ so let $\phi: V\to B$ be an $H$-isomorphism  and for every $1\leq i \leq d,$ set $b_i:=z_i^\phi.$ For every $1\leq i\leq t^*$, let $X_i:=W_iH^*$ and, for every $1\leq j\leq d$, let $Y_j:=A (H^*)^{b_j}.$ It follows from Lemma \ref{interKM}, that 
$$(X_1\cap\dots\cap X_{t^*})\cap (Y_1\cap\dots\cap Y_d)=(UH^*)\cap (A C_{H^*}(Z))=UC_{H^*}(Z).\quad \qedhere$$
\end{proof}

\section{Bounded intersection and bounded chief factors properties}

In this section we will prove that in the class of prosolvable groups the bounded intersection property and the bounded chief factors property are equivalent.

\begin{proof}[Proof of Theorem \ref{thuno}]
	We proceed by induction on $|G|$. We may assume $\frat(G)=1$. By Lemma \ref{corona}, we can assume there exists an irreducible $G$-module $V$ such that $C=R\times D$ with $1\neq D\cong_GV^t$, where $C=C_G(V)$, $R=R_G(V)$ and $t=\delta_G(V)$. Assume that $H$ is a maximal intersection in $G$ and that every $U\in\mathcal V_H$ is a $\gamma$-module. We want to prove that there exists a family $M_1,\dots,M_n$ of maximal subgroups of $G$ such that $H=M_1\cap \cdots \cap M_n$ and $|G:M_1|\cdots |G:M_n|\leq |G:H|^{\gamma+1}.$
	
	First assume $D\le H.$ If $M\in\mathcal M_V,$ then $MD=G,$ hence $H\not\le M$ and consequently $V\notin\mathcal V_H.$ We may then work module $D$ and conclude by induction.
	Hence we may assume $D\not\le H$. We can write
	$$H=X_1\cap\ldots\cap X_{\rho}\cap Y_1\cap\ldots\cap Y_{\sigma},$$
	where  $X_1,\dots,X_\rho$ are maximal subgroups not containing $D$ 
	and $Y_1,\dots,Y_\sigma$ are maximal subgroups containing $D.$
	Notice also that $\{X_1,\dots,X_\rho\}\subseteq \mathcal M_V,$ and consequently
$V\in \mathcal V_H.$		We define 
$$X:=X_1\cap\ldots\cap X_{\rho} \text{ and } Y:=Y_1\cap\ldots\cap Y_{\sigma}.$$
	
	By Lemma \ref{sotto}, 
	$R\le X_i$ for every $i$, in particular  $R\le X$. By the properties of the crowns, there exists $K\le G$ such that
	$$G/R=C/R\rtimes K/R\cong V^t\rtimes K/R,$$ 
	where $V$ can be seen as an irreducible  $K$-module with $C_K(V)=R.$ Note that $X/R$ is an intersection of maximal subgroups of $G/R$ supplementing $C/R\cong_{G}V^t$, so we may apply Theorem \ref{impor}: there exists an intersection $T/R$ of maximal $G$-submodules of $C/R,$ a conjugate $K^*$ of $K$ in $G$ and an $\End_G(V)$-subspace $Z$ of $V$ such that
	$$X/R=T/R \rtimes C_{K^*}(Z)/R.$$ Since $V\in \mathcal V_H,$ there is an $\End_G(V)$-subspace $Z^*$ of $Z$ such that $\dim_{\End_G(V)} Z^* \leq \gamma$ and
	$$C_{K^*}(Z)=C_{K^*}(Z^*)\cap (\cap_{M\in \mathcal M_Z}M)$$ where $\mathcal M_Z$ is the set of the maximal subgroups of $K^*$ containing $C_{K^*}(Z).$ We have $C/T\cong_G V^{t^*}$ for some positive integer $t^*$ and, by Theorem \ref{possibile}, there exists $\alpha\leq t^*+\gamma$ and a family $\tilde X_1,\dots,\tilde X_\alpha$ of maximal subgroups
	of $G$ containing $R$ and supplementing $C$ such that
	$$\tilde X_1/R\cap \dots \cap \tilde X_\alpha/R= T/R \rtimes C_{K^*}(Z^*)/R. \ \text { i.e. }\ \tilde X_1\cap \dots \cap \tilde X_\alpha=TC_{K^*}(Z^*).$$
	Assume now that $\mathcal M_Z=\{M_1,\dots,M_\beta\}:$ for every $1\leq i\leq \beta,$ there exists a maximal subgroup $\tilde Y_i$ containing $C$ such that $\tilde Y_i/R = C/R \rtimes M_i/R.$ Notice that 
			$$\left(\cap_{1\leq i \leq \alpha} {\tilde X_i}\right)\cap \left(\cap_{1\leq j \leq \beta} {\tilde Y_j}\right) =\left(T C_{K^*}(Z^*)\right)\cap \left(C  \left(\cap_{1\leq j\leq \beta} M_j\right)\right)= T C_{K^*}(Z)=X
			,$$ hence
		$$\tilde X_1\cap \dots \cap \tilde X_\alpha \cap \tilde Y_1\cap \dots \cap \tilde Y_\beta =X.$$
		Let
			$$\tilde X=\tilde X_1\cap \dots \cap \tilde X_\alpha, \quad \tilde Y = \tilde Y_1\cap \dots \cap \tilde Y_\beta \cap Y_1\cap \dots\cap Y_\sigma.$$
			We have $\tilde X \cap \tilde Y = X \cap Y = H.$ Notice that $\tilde Y$ is an intersection of maximal subgroups of $G$ containing $D$, so by induction there exists $\tau$ maximal subgroups $Q_1,\dots,Q_\tau$ of $G$ such that
			$$\tilde Y=Q_1\cap\dots\cap Q_\tau \quad \text { and } \quad \prod_{1\leq j\leq \tau}|G:Q_j|\leq |G:\tilde Y|^{\gamma+1}.$$
		Define $D^*=D\cap X$. Thus, by the Dedekind Law, we have $T=DR\cap \tilde X=(D\cap \tilde X)R=D^*R$, and note that
			 $$V^{t-t^*}\cong_G T/R \cong_G D^*R/R\cong_G D^*/(D^*\cap R)\cong_G D^*.$$
Moreover  $\tilde Y\ge D$ implies $\tilde X\tilde Y\geq \tilde XD$, and so $|\tilde X\tilde Y|\ge |\tilde XD|$. 
 Hence,
	$$|\tilde X||\tilde Y|/|\tilde X\cap \tilde Y|\geq |\tilde X||D|/|\tilde X\cap D|,$$
	so
	$$|\tilde Y:\tilde X\cap \tilde Y|\ge |D:\tilde X\cap D|=|D:D^*|=|V|^{t^*}.$$ We have
	$$H=\tilde X_1\cap \dots \cap \tilde X_\alpha\cap Q_1\cap \dots \cap Q_\tau$$
and, since  $|G:X_i|=|V|$ for every $i\in \{1,\dots,\alpha\}$, $\alpha \leq t^*+\gamma$ and $t^*\geq 1,$
	\begin{equation*}
	\begin{split}
	\prod_{i=1}^{\alpha}|G:\tilde X_i|\prod_{j=1}^{\tau}|G:Q_j|&\le |V|^{t^*+\gamma}|G:\tilde Y|^{1+\gamma}\le (|G:\tilde Y||V|^{t^*})^{1+\gamma}\\
	&\le (|G:\tilde Y||\tilde Y:\tilde X\cap \tilde Y|)^{1+\gamma}\le |G:\tilde X\cap \tilde Y|^{1+\gamma}=|G:H|^{1+\gamma},
	\end{split}
	\end{equation*}
	and\ the theorem follows.
\end{proof}

\begin{proof}[Proof of Theorem \ref{due}]
Assume that $V$ is an irreducible $G$-module $G$-isomorphic to a non-Frattini chief factor of $G$ and let $H:=G/C_G(V).$ The semidirect product $\Gamma  =V\rtimes H$ is an epimorphic image of $G$, so it has the $\eta$-intersection property. Let $F=\End_G(V)$ and let $W$ be an $F$-subspace of $V.$ By Theorem \ref{possibile}, $X=C_H(W)$ is an intersection of maximal subgroups of $\Gamma,$ so there exists $n$ maximal subgroups $X_1,\dots,X_n$ of $\Gamma$ such that $X=X_1\cap\dots\cap X_n$ and $|\Gamma:X_1|\cdots |\Gamma:X_n|\leq |\Gamma:X|^\eta.$
Let $\overline X$ be the intersection of the maximal subgroups of $H$ containing $X.$ We may assume that $X_i$ contains $V$ if and only if $i>m:$ this means that for every $j\leq m$ there exists $v_j\in V$ such that $X_j=H^{v_j}$ while $X_{m+1}\cap\dots \cap X_n\geq V\rtimes \overline X.$ In particular 
$$C_H(v_1,\dots,v_m)\cap \overline X = H\cap H^{v_1}\cap \dots \cap H^{v_m}\cap  (V\rtimes \overline X )\subseteq X,$$
which implies that if $W^*$ is the $F$-subspace of $V$ spanned by $v_1,\dots,v_m$, then $C_H(W)=C_H(W^*)\cap (\cap_{M\in \mathcal M_W}M).$ By \cite{palfy} and \cite{wolf}, $|\Gamma|\leq |V|^c,$ so
$$|V|^m=\prod_{1\leq i \leq m}|\Gamma:X_i|\leq \prod_{1\leq i \leq n}|\Gamma:X_i|
\leq |\Gamma:X|^\eta \leq |\Gamma|^\eta \leq |V|^{c\cdot \eta}.$$ Hence $\dim_F (W^*)\leq m\leq \lfloor \eta\cdot c\rfloor.$
\end{proof} 

It remains to prove Corollary \ref{fittingamma}. For this purpose, we need the following observation:

\begin{lemma}\label{facile}Assume that $G$ is a finite solvable group with nilpotent derived subgroup. If $V$ is an irreducible $G$-module $G$-isomorphic to a non-Frattini chief factor of $G,$ then $\dim_{\End_GV}V=1.$
\end{lemma}

\begin{proof}
We may assume $\frat (G)=1.$ This means that $G=M\rtimes H$ where $H$ is abelian and $M=V_1\times \cdots \times V_u$ is the direct product
of $u$ irreducible non trivial $H$-modules $V_1,\dots,V_u.$ Let $F_i=\End_H(V_i)=\End_G(V_i)$: for each $i\in \{1,\dots,u\},$ $V_i$ is an absolutely irreducible $F_iH$-module
so $\dim_{F_i}V_i=1.$ Now assume that $A$ is a nontrivial irreducible $G$-module $G$-isomorphic to a complemented chief factor of $G$:
it must be $A\cong_G V_i$ for some $i$, so $|\End_G(A)|=|F_i|=|V_i|=|A|.$
\end{proof}

\begin{proof}[Proof of Corollary \ref{fittingamma}]
Let $G$ be a finitely generated prosolvable group with pronilpotent derived subgroup. By Lemma \ref{facile}, every non-Frattini chief factor of $G$ is  a 1-module, so, by Theorem \ref{thuno}, $G$ has the 2-intersection property. The conclusion follows from Proposition \ref{propo}.
\end{proof}

\section{Prosupersolvable groups}\label{super}

Recall that the Dirichlet Theorem on arithmetic progressions states that for any two positive coprime integers $a$ and $b$, there exist infinitely many primes which are congruent to $a$ modulo $b$. In particular, the arithmetic progression $\{1+r\cdot 2^n\mid r\in\N\}$
contains infinitely many primes. This implies that there exists a strictly ascending sequence $\{p_n\}_{n\in\N}$ of primes with the property that $2^n$ divides $p_n-1$ and $p_{n+1} > 2^n\cdot p_1\cdots p_n.$

Let $V_m$ be a 1-dimensional vector space over $\F_m$, where $\F_m$ is the field with $p_m$ elements, and $H_n:=\langle x_n\rangle$ be a cyclic group of order $2^n$ . We can define an action of $H_n$ on $V_m$, for every $m\le n$, as follows: if $v\in V_m$, then $v^{x_n}:=\zeta_mv$, where $\zeta_m$ is an element of order $2^m$ in $\F_m^*$.
Note that $C_{H_n}(V_m)=\langle x_n^{2^m}\rangle$. Consider the following finite supersolvable group:
$$G_n=(V_1\times\ldots\times V_n)\rtimes H_n.$$
Let us describe the maximal subgroups of $G_n$. Let $W=V_1\times\ldots\times V_n$ and for every $1\leq i \leq n,$ set $W_i:=V_1\times\ldots\times V_{i-1}\times V_{i+1}\times\ldots\times V_n.$
First note that $M=W\rtimes\langle x_n^2\rangle$ is the unique maximal subgroup of $G$ containing $W$. 
The other maximal subgroups of $G_n$ are the semidirect products $W_i\rtimes H_n^v$ for $i\in\{1,\dots,n\}$ and $v_i\in V$, so we have precisely
$p_i$ maximal subgroups of index $p_i$ for every $1\le i\le n$.
Consider now the maximal subgroups $M_1=W_i\rtimes H_n^{v_1}$ and $M_2:=W_i\rtimes H_n^{v_2}$ with $v_1,v_2\in V_i$. If $v_1\neq v_2$, then $M_1\neq M_2$ and, by  Lemma \ref{interKM},  $M_1\cap M_2=W_i\rtimes\langle x_n^{2^i}\rangle,$ since $\langle x_n^{2^i}\rangle=C_{H_n}(v)$ for every $0\neq v \in V_i.$ These considerations imply that a subgroup $H$ of $G_n$ can be expressed as intersection of maximal subgroups of $G$ if and only if it is conjugated to one of the following subgroups:
\begin{enumerate}
	\item $X_J=\left(\prod_{j\notin J}V_j\right) \rtimes H_n,$ where $J$ is a non-empty subset of $\{1,\dots,n\};$ note that $|G_n:X_J|=\prod_{j\in J}p_j.$
	\item $Y_J=\left(\prod_{j\notin J}V_j\right) \rtimes \langle x_n^2 \rangle,$  where $J$ is a non-empty subset of $\{1,\dots,n\};$ note that $|G_n:Y_J|=2\cdot \prod_{j\in J}p_j.$
	\item $Z_{J,i}=\left(\prod_{j\notin J}V_j)\right) \rtimes \langle x_n^{2^i}  \rangle,$ where $i\in \{2,\dots,n\}$ and $J$ is a subset of $\{1,\dots,n\}$ containing $i$. Note that $|G_n:Y_{Z,i}|=2^i\cdot \prod_{j\in J}p_j.$ 
\end{enumerate}

By \cite[Theorem 1.5]{AL}, if $K$ is a subgroup of $G$ with $\mu(K,G_n)\ne 0$, then there exists a family of maximal subgroups $M_1,\ldots,M_t$ of $G_n$ such that $K=M_1\cap \cdots\cap M_t$ and $|G_n:K|=|G_n:M_1|\cdots|G_n:M_t|$. The subgroups $Z_{J,i}$ do not have this property, hence $\mu(Z_{J,i},G_n)=0$ for all the possible choices of $J$ and $i.$ It follows that, for $n\geq 2$,
$$\begin{aligned}\tilde \gamma_{2^n\cdot p_1\cdots p_n}(G_n)&=2^n-1+2^n+(n-2)\cdot 2^{n-1}=2^{n-1}\cdot (n+2)-1,\\\tilde \beta_{2^n\cdot p_1\cdots p_n}(G_n)&\leq 2^{n+1}-1
\end{aligned}$$ and consequently
$$\frac{\tilde\beta_{2^n\cdot p_1\cdots p_n}(G_n)}{\tilde\gamma_{2^n\cdot p_1\cdots p_n}(G_n)}\leq \frac{4}{n+2}.$$ Now consider the inverse limit
$G=\varprojlim_nG_n.$ Note that $G$ is a 2-generated prosupersolvable group, with $G\cong (\prod_i V_i)\rtimes \mathbb Z_2,$ being $\mathbb Z_2$ the group of the 2-adic integer. Let $n^*=2^n\cdot p_1\cdots p_n.$ The condition
$p_{n+1} > 2^n\cdot p_1\cdots p_n$
for every $n\in \mathbb N,$ implies that if $H$ is an open subgroup of $G$ with $|G:H|>n^*,$ then $\left(\prod_{j>n}V_j\right)\rtimes \mathbb Z_2^{2^n}\leq H$ and consequently $\tilde\beta_{n^*}(G)=\tilde\beta_{n^*}(G_n)$ and
$\tilde\gamma_{n^*}(G)=\tilde\gamma_{n^*}(G_n).$ This implies
$$\liminf_{n\to \infty}\frac{\tilde \beta_n(G)}{\tilde \gamma_n(G)}\leq\liminf_{n\to \infty}\frac{\tilde\beta_{n^*}(G)}{\tilde\gamma_{n^*}(G)}
\leq \lim_{n\to \infty}\frac{4}{n+2}=0.
$$

\end{document}